\documentclass[12pt]{amsart}

\usepackage{amsmath,amsfonts,amssymb,amscd}
\theoremstyle{definition}

\def\R{\mathbb{R}}

\def\g{\mathfrak{g}}
\def\h{\mathfrak{h}}

\newcommand{\CalC}{\mathcal{C}}
\newcommand{\CalD}{\mathcal{D}}

\newcommand{\CalI}{\mathcal{I}}

\newcommand{\CalP}{\mathcal{P}}

\newcommand{\CalV}{\mathcal{V}}

\newcommand{\st}{\; \; | \; \;}

\newtheorem{thm}{Theorem}[section]
\newtheorem{cor}[thm]{Corollary}
\newtheorem{prop}[thm]{Proposition}
\newtheorem{lem}[thm]{Lemma}
\newtheorem{defn}[thm]{Definition}

\newcommand{\rank}{\operatorname{rank}}

\newcommand{\Hom}{\operatorname{Hom}}

\newcommand{\Ad}{\operatorname{Ad}}

\title[Representation Variety of  surface groups]{{Representation variety of surface groups}}

\author{Krishna Kishore}

\begin{document}

\maketitle

\begin{abstract}
We give an exact formula for the dimension of the variety of homomorphisms from $S_g$ to \textit{any} semisimple real algebraic group, where $S_g$ is a surface group of genus $g \geq 2$.
\end{abstract}

\section{Introduction}\label{intro}

Let $\Gamma$ be a finitely generated group and $G$ \textit{any} linear algebraic group defined over $\R$. The set of homomorphisms $\Hom(\Gamma,G(\R))$ coincides with the real points of the \textit{representation variety} $X_{\Gamma,G}:= \Hom(\Gamma,G)$. (We note here that by a \textit{variety} we mean an affine scheme of finite type over $\R$. In particular, we do not assume that it is irreducible or reduced.) Let $S_g$ be a surface group of genus $g \geq 2$. Throughout the paper we assume genus $g \geq 2$. The main goal of the paper is to show that the dimension of $X_{S_g,G}$ is roughly of the order $ (2g-1)\dim G$, where $G$ is \textit{any} semisimple real algebraic group. Before stating the main result, let us fix some notation. 

A cocompact oriented Fuchsian group $\Gamma$ admits a presentation of the following kind: consider non-negative integers $m$ and $g$ and integers $d_1,\ldots,d_m$ greater than or equal to $2$, such that the Euler characteristic
\begin{equation*} \label{cha}
\chi(\Gamma) : = 2-2g-\sum_{i=1}^{m}(1 - d_i^{-1})
\end{equation*}
\noindent is negative. For some choice of \textit{such} $m$, $g$, and $d_i$, $\Gamma$ has a presentation of the following kind:
\begin{equation*}\label{pre}
\begin{split}
\Gamma  :=  \langle x_1,\ldots,x_m,y_1,\ldots,y_g,z_1,\ldots, & z_g  \; \vert \;  x_1^{d_1},\ldots,x_m^{d_m}, \\
 & x_1 \ldots x_m[y_1,z_1]\ldots[y_g,z_g] \rangle. 
 \end{split}
\end{equation*}
\noindent In particular, a surface group $S_g$ of genus $g \geq 2$ admits a presentation of the following kind:
\begin{equation*}\label{surface_group_presentation}
\Gamma  :=  \langle y_1,\ldots,y_g,z_1,\ldots,  z_g  \st  [y_1,z_1]\ldots[y_g,z_g] \rangle. 
\end{equation*} 

Let $\Ad$ denote the adjoint representation of $G$ in its Lie algebra $\g$ and let $(\Ad \circ \rho)^{*}$ be the coadjoint representation of $\Gamma$, where the adjoint representation of $\Gamma$ is $\Ad \circ \rho : \Gamma \to G \to GL(\g)$.  The Zariski tangent space at any point $\rho \in X_{\Gamma,G}$ is given by the space of $1$-cocycles $Z^1(\Gamma, \Ad \circ \rho)$ and its dimension is given by the following formula \cite{We}:
\[
\dim   Z^1(\Gamma,\Ad \circ \rho) := (2g-1)\dim G + \dim (g^*)^\Gamma + \sum_{j=1}^m (\dim \g - \dim \g^{\langle x_i \rangle}),
\]
where $(\g^{*})^\Gamma$ denotes the $\Gamma$-invariant vectors under the coadjoint representation $(\Ad \circ \rho)^*$ of $\Gamma$. 
If $\Gamma$ is a surface group, the last summand in the above formula vanishes, and the formula assumes a simpler form:
\begin{equation}\label{dim_tangent_space}
\dim   Z^1(\Gamma,\Ad \circ \rho) := (2g-1)\dim G + \dim (g^*)^\Gamma .
\end{equation}
One can consider $\rho$ to be the trivial representation, whence the second summand in the formula reduces to $\dim g^*= \dim \g = \dim G$. Therefore the dimension of any irreducible component is bounded above by $(2g-1) \dim G + \dim G = 2g \dim G$: 
\begin{equation}\label{upper_bound}
\dim X_{S_g,G} \leq 2 g \dim G.	
\end{equation}
A more precise estimate of $\dim (g^*)^\Gamma$, namely that it is of asymptotic order $O(\rank G)$ follows from Lemma $1.2$ of \cite{LL}. But we prefer to be content with the above crude estimate in order to arrive at a simpler \textit{form} of the bounds of $X_{S_g,G}$ namely that the difference between the lower estimate and the upper estimate of $\dim X_{S_g, G}$ is at most $\dim G$ (Theorem \ref{main_result_intro}).

In this paper we prove existence of injective homomorphisms $\rho: S_g \to G(\R)$ with Zariski-dense image and that they are nonsingular points of the variety $X_{S_g,G(\R)}$. As it turns out, the dimension of the unique irreducible component to which the nonsigular point belongs has dimension $(2g-1)\dim G$: 

\begin{prop}\label{dense_hom_intro}
Let $G$ be a semisimple real algebraic group. Let $F_2$ be a free group of rank $2$. Consider the subset $\CalD \subset \Hom(F_2, G)$ consisting of homomorphisms $F_2 \to G $ that are injective and has Zariski-dense image in $G(\R)$. Then the set $\CalD$ is generic in the real algebraic group $\Hom(F_2, G(\R))$.
\end{prop}

\noindent From the proposition it easily follows that the dimension of the representation variety $X_{S_g, G}$ is bounded below by $(2g-1) \dim G$. On the other hand we show that irreducible components containing no Zariski-dense representations, i.e. representations with Zariski-dense image, can be ignored in the computation of the dimesion $\dim X_{S_g,G}$.

\begin{prop} \label{proper_remove_intro}
Let $G$ be an almost-simple real algebraic group. Consider the representation variety $X := X_{S_{g},G}(\R)$ where $S_g$ is a surface group of genus $g \geq 2$. Suppose $\CalC$ is an irreducible component of $X$ containing homomorphisms $\phi : S_g \to G(\R)$ such that the Zariski-closure of $\phi(S_g)$ is a proper closed subgroup of $G(\R)$. Suppose $\CalD$ be an irreducible component of $X$ containing at least one homomorphism with Zariski-dense image.
Then  $\dim \CalD - \dim \CalC \geq 0$, provided $ \rank G \geq 2$.
\end{prop}

Consequently, from these propositions together with the upper bound estimate (\ref{upper_bound}) it follows that, we have the following exact formula:

\begin{thm}\label{main_result_intro}
For every surface group $S_g$, where $g \geq 2$, and for every almost-simple real algebraic group $G$, 
\begin{equation*}
\dim   X_{S_g, G} =  (2g-1) \dim G .
\end{equation*}
\end{thm}

It is worthwhile to compare the results in this paper with those in \cite{LL} and \cite{Ki}.  While Theorem \ref{main_result_intro} provides an \textit{exact formula} the dimension of the representation variety $X_{\Gamma,G}$  but only for surface groups of genus $g \geq 2$,  results in \cite{LL} and \cite{Ki} offers an \textit{estimate} for all Fuchsian groups $g \geq 0$ but only for the representation variety $X_{\Gamma,G}^{\textrm{epi}}$  (Let us note here that $X_{\Gamma,G}^{\textrm{epi}}$ is the closed subset of $X_{\Gamma,G}$ consisting of homomorphisms $\Gamma \to G(\R)$ with Zariski-dense image.)

Moreover it may be worthwhile to compare the various techniques used. In \cite{LL} and in \cite{Ki} the technique based on the deformation theory of Weil is used in an essential manner. The difficult aspect there was to establish a lower bound on $X_{\Gamma,G}$. So to obtain the results therein, analysis of various subgroups of the simple algebraic group under consideration was carried out. On the other hand the technique used in this paper obviates the necessity for such analysis and is based on the results of \cite{BGGT}.

The paper is organized as follows. In \S \ref{locus} we prove a result on the existence of Zariski-dense homomorphisms.  In \S \ref{mainresult} we provide an estimate on the dimension of $X_{S_g, G}$ using the results of the previous sections. In $\S \ref{codim_proper}$ we first prove that the irreducible components of $X_{S_g,G}$ not containing those represnetations $S_g \to G$ with Zariski-dense image can be ignored in the computation of $\dim X_{S_g,G}$, then we prove Theorem \ref{main_result_intro}. In \S \ref{finalremarks} we comment on extending the approach taken in this paper to arbitrary Fuchsian groups of genus $g \geq 0$.

\section{Existence of dense homomorphisms}\label{locus}

\noindent In this section we prove the existence of dense homomorphims $S_g \to G$ from a surface group $S_g$ of genus $\mathbf{g \geq 2}$ to a connected semisimple real algebraic group $G$. Before we begin with results proper, let us recall some notions.
\begin{defn}
Let $X$ be a topological space. A subset $A \subset X$ is called \textit{meagre} if it is a union of countably many nowhere dense subsets of $X$. The complement of a meager set in $X$ is called \textit{generic}.
\end{defn}
In this context, let us recall a well-known result of Baire (Baire category theorem):
\begin{thm}\label{baire}
(Baire) Let $X$ be a complete metric space and let $E_1, E_2, \ldots$ be an at most countable sequence of subsets of $X$.  If the union $\bigcup_n E_n$ of $E_n$ contains a ball $B$, then at least one of the $E_n$ is dense in a sub-ball of $B$.
\end{thm}

Also note that the representation variety $\Hom(F_n, G(\R))$ is a real algebraic \textit{group} (in the Zariski topology) and also a smooth real manifold (in the Euclidean topology).

\begin{lem}\label{inj}
Let $G$ be a semisimple real algebraic group. Let $F_2$ be a free group of rank $2$. Consider the subset $\CalV \subset \Hom(F_2, G(\R))$ consisting of homomorphisms $F_2 \to G(\R) $ that maps $w$ to $1$ for some nontrivial word $w \in F_2$. Then $\CalV$ is a meagre subset of $G(\R)\times G(\R)$.
\end{lem}
\begin{proof}
Let $w \in F_2$, and consider the word map $f_w: G(\R) \times G(\R) \to G(\R)$ defined by $(a,b) \mapsto w(a,b)$. By a theorem of Borel \cite[Theorem 1]{Bo}, the map $f_w$ is dominant, hence the inverse image $f_w^{-1}(\{1\})$ is a \textit{proper closed subset} of $G(\R) \times G(\R)$. As any homomorphism $\phi: F_2 \to G(\R)$ is determined by the images of the generators of $F_2$, it follows that the subset $\CalV_w  \subset \Hom(F_n, G(\R))$ consisting of homomorphisms that map the word $w$ to $1$ is a \textit{proper closed subvariety}. On the other hand any proper closed subvariety is nowhere dense in the real smooth manifold $G(\R) \times G(\R)$. Since the set of reduced words in $F_2$ is countable, it follows that the union of \textit{word-varieties} $\CalV$ consisting of homomorphisms $F_2 \to G(\R) $ that maps $w$ to $1$ for some nontrivial word $w \in F_2$, being  a countable union $\bigcup_{w \in F_2} \CalV_w$ of nowhere dense sets, is countable too, and therefore meagre by the Baire category theorem (Theorem \ref{baire}.)
\end{proof}

In the next proposition we need a result on the existence of \textit{dense-pairs} $(a,b) \in G(\R) \times G(\R)$ such that the subgroup generated by $a,b$ is dense in $G(\R)$ \cite{BGGT}. For sake of convenience, we cite it:

\begin{thm}\label{BGGT}
(Breuillard, Green, Gurnalick, and Tao)
Suppose that $G(k)$ is a semisimple algebraic group over a field $k$ of characteristic zero, and that $w,w' \in F_2$ are noncommuting words. Then
\[
X := \{ (a, b) \in G(k) \times G(k) \st \overline{w(a, b),w`(a, b)}= G \}
\]
is an open subvariety of $G \times G$ defined over $k$ and $X(k)$ is non-empty.
\end{thm}

Now we prove that the homomorphisms $F_2 \to G(\R)$	 with Zariski-dense image is a generic set.

\begin{prop}\label{dense_hom}
Let $G(\R)$ be a semisimple real algebraic group. Let $F_2$ be a free group of rank $2$. Consider the subset $\CalD \subset \Hom(F_2, G(\R))$ consisting of homomorphisms $F_2 \to G(\R) $ that are injective and has Zariski-dense image in $G(\R)$. Then $\CalD$ is generic in the real manifold $\Hom(F_2, G(\R))$. 
\end{prop}
\begin{proof}
From Lemma \ref{inj}, it follows that the subset $\CalI$ of $\Hom(F_2,G(\R))$ consisting of homomorphisms $\phi: F_2 \to G(\R)$ that are injective is contained in the complement of a meagre set, in other words, in a generic set. Note that the image of an injective homomorphism $\phi: F_2 \to G(\R)$ is a free subgroup of rank $2$ contained in $G(\R)$. By the theorem of Breuillard, Green,Gurnalick, and Tao \cite[Theorem 4.1]{BGGT}, it follows that any injective homomorphism $\phi: F_2 \to G(\R)$ has Zariski-dense image in $G(\R)$, and that the set of such homomorphisms is nonempty and open in the real algebraic group $\Hom(F_2,G(\R))$.
\end{proof}

Before beginning the main result of the section, we recall a well known notion:

\begin{defn}
A group $G$ is \textit{fully-residually free} if for any given finite subset $X \subset G$ such that $X$ does not contain the identity (of $G$), there exists a surjective homomorphism $f: G \to F_n$ from $G$ to the free group $F_n$ of rank $n \geq 1$, such that $f(x) \neq 1$ for all $x \in X$.
\end{defn}

\begin{thm}\label{dense_hom_existence}
Let $G$ be a semisimple real algebraic group. Let $X_{S_g,G}$ denote the representation variety $\Hom(S_g,G)$ (where, recall, $g \geq 2$.) The set of homomorphisms $S_g \to G(\R)$ with Zariski-dense image is nonempty.
\end{thm}

\begin{proof}
It is well known that surface groups $S_g$ of genus $g \geq 2$ are fully residually-free. Hence there exists a surjective homomorphism $\phi: S_g \to F_n$, where $g,n \geq 2$ with the property that each of the $2g$ generators of $S_g$ are not mapped to the identity of $F_n$. 

Also, it is well known that a free group of rank $2$ contains any free group of rank $m \geq 2$. It can be easily seen that Proposition \ref{dense_hom} continues to hold for any free group $F_n$ with $n \geq 2$. Consequently, the set of homomorphisms $F_n \to G(\R)$ with Zariski-dense image is \textit{real-dense} in the real algebraic \textit{group} $\Hom(F_n, G(\R))$. Composing with the morphism $\phi$ of the previous paragraph, we obtain a morphism $S_g \to G(\R)$ with Zariski-dense image.
\end{proof}

\section{Dimension estimate of $X_{S_g,G}$}\label{mainresult}

In this section we establish an estimate on the dimension of $X_{S_g,G}$. The lower bound of this estimate shall be used in the next section to give an exact formula for $X_{S_g,G}$, namely that $\dim X_{S_g,G} = (2g-1) \dim G$.

\begin{lem}\label{non_singularity}
Let $\rho: \Gamma \to G(\R)$ be a representation of a finitely generated group $\Gamma$ into a semisimple real algebraic group $G$. Let $(\g^*)^{\Gamma}$ be the space of  $\Gamma$-invariant vectors considered under the coadjoint representation  $(\Ad \circ \rho )^{*}$. Suppose $\rho(\Gamma)$ is Zariski-dense in $G(\R)$. Then $\rho$ is a nonsingular point in $X_{\Gamma,G}$. In particular, it belongs to a unique irreducible component of $X_{\Gamma,G}$.
\end{lem}

\begin{proof} 
First note that the adjoint representation in $\g$ is a self-dual $G$-representation, via the Killing form. It follow then that $(\g^*)^\Gamma = \g^\Gamma = \g^G$  (the latter equality is due to the hypothesis that $\rho(\Gamma)$ is dense in $G$.) The dimension of $\g^G$ is equal to the dimension of the centralizer of $G$ in $G$, which is finite because $G$ is semisimple. Therefore $\dim \g^G = 0$.  On the other hand, if the coadjoint representation $(Ad \circ \rho)^*$ of $\Gamma$ in $\g$ has no $\Gamma$-invariant vectors then $\rho$ is a nonsingular point in $X_{\Gamma,G}$ \cite{We}.
It follows then that the representation $\rho$ is nonsingular in $X_{\Gamma,G}$ and, in particular, belongs to a unique irreducible component of $X_{\Gamma,G}$.
\end{proof}

\begin{thm}\label{main_result_section_3}
For every surface group $S_g$, where $g \geq 2$, and for every connected semisimple real algebraic group $G$, 
\begin{equation*}
 (2g-1) \dim G \leq \dim   X_{S_g, G}  \leq 2g \dim G.
\end{equation*}
\end{thm}

\begin{proof}
The upper bound follows from the estimate (\ref{upper_bound}). It remains to show the lower bound. By Theorem \ref{dense_hom_existence} there exists a homomorphsm $\rho: S_g \to G(\R)$ with Zariski-dense image. By Lemma \ref{non_singularity} the representation $\rho$ is a nonsingular point, hence it belongs to the unique irreducible component (containing $\rho$) of $X_{S_g, G(\R)}$.

On the other hand, the proof of the Lemma \ref{non_singularity} show that the dimension of the space $(\g^*)^{S_g}$ is $0$. Now, the lower bound follows from the formula (\ref{dim_tangent_space}).
\end{proof}

We may get an exact formula with out the results of the next section in the case where $G$ is simply connected, as the folllowing corollary shows:

\begin{cor}\label{exact_dim}
For every surface group $S_g$, where $g > 1$, and for any simply-connected semisimple real algebraic group $G$, 
\begin{equation*}
\dim   X_{S_g, G}  =  (2g-1) \dim G.
\end{equation*}
\end{cor}
\begin{proof}
If $G$ is simply-connected then the representation variety $X_{S_g, G}$ is an irreducible variety. Now, the corollary follows the above theorem (more precisely, from the first paragraph of the theorem \ref{main_result_section_3}).
\end{proof}

\section{Codimension of Proper Subvarieties}\label{codim_proper}

Let $\CalP$ denote  an irreducible component of $X_{S_g,G(\R)}$ consisting of representations $\rho: S_g \to G(\R)$ with image not Zariski-dense in $G$. Let $\CalD$ be an irreducible component of $X_{S_g,G(\R)}$ consisting of representations $\rho: S_g \to G(\R)$ with Zariski-dense image in $G(\R)$ (such a component exists due to Theorem \ref{dense_hom_existence}.) In this section we show that $\dim \CalP \leq \dim \CalD$. We cite a result of Larsen and Lubotzky \cite[Prop. 3.1.]{LL}, which will be used below:

\begin{prop}\label{remove_parabolics}
(Larsen, Lubotzky) Let $G$ be a linear algebraic group over $\R$ and $H \subset G$ 
a closed subgroup such that $G(\R)/H(\R)$ is compact. Let $\Gamma$ be a cocompact oriented Fuchsian group of genus $g \geq 0$. Let $\CalC$ denote an irreducible component of $X_{\Gamma,H}$. The condition on $\rho \in X_{\Gamma,G(\R)}$ that
$\rho$ is not contained in any $G(\R)$-conjugate of $C(\R)$ is open in the real topology.
\end{prop}

Before we prove the main result of the paper, we establish an elementary result.

\begin{lem}\label{codim_proper}
Let $G$ be an almost-simple real algebraic group, and let $H$ be a maximal proper closed subgroup.
Then $\dim G - \dim H \geq \rank G$.
\end{lem}

\begin{proof}
Consider the natural action of $G$ on $G/H$ by left translation. As $G$ is almost-simple and $H$ is proper, the action is effective. In particular the restricted action to a maximal torus $T$ of $G$ is also effective. It follows then that the dimension of generic orbit of $T$ is $\dim T$. The lemma follows from the following inequalities:
$$
\dim G - \dim H = \dim G/H \geq \dim T = \rank G.
$$
\end{proof}

Now we prove the main result of this section.

\begin{prop} \label{proper_remove}
Let $G$ be an almost-simple real algebraic group. Consider the representation variety $X := X_{S_{g},G}(\R)$ where $S_g$ is a surface group of genus $g \geq 2$. Suppose $\CalC$ is an irreducible component of $X$ containing homomorphisms $\phi : S_g \to G(\R)$ such that the Zariski-closure of $\phi(S_g)$ is a proper closed subgroup of $G(\R)$. Suppose $\CalD$ be an irreducible component of $X$ containing at least one homomorphism with Zariski-dense image.
Then  $\dim \CalD - \dim \CalC \geq 0$, provided $ \rank G \geq 2$.
\end{prop}

\begin{proof}
As any nonempty open set in real topology is Zariski-dense in $X_{S_g, G(\R)}$, it follows from Proposition \ref{remove_parabolics} that $\CalD$ contains a point $\rho \in X_{S_g,G(\R)}$ with image contained in a reductive subgroup of $G(\R)$, since a maximal proper closed subgroup $H$ of $G$ is such that either $H^\circ$ is reductive or $H$ is parabolic \cite[Theorem 30.4]{Hu}. Hence it suffices to prove the proposition assuming that $\CalC$ contains a point $\rho: S_g \to G$ with image contained in a proper closed reductive subgroup $H$ of $G$. Let $\h$ be the Lie algebra of $H$. By a result of Larsen and Lubotzky \cite[Prop. 2.1]{LL}, it follows that the dimension of the space of $1$-cocyles $Z^1(S_g, \h)$ at the point $\rho$, considered as a point in $X_{S_g,H} \subset X_{S_g,G}$, is given by 
\begin{equation}\label{proper_dim}
\dim Z^1(S_g,\h) \leq (2g-1) \dim H + 2g + \rank H.
\end{equation}
Therefore the dimension of the irreducible component $\CalC$ is bounded above by $(2g-1) \dim K + 2g + \rank G$ where $K$ is a maximal proper closed reductive subgroup of $G$. On the other hand by Theorem \ref{main_result} the dimension of $\CalD$, which contains a representation $S_{g} \to G(\R)$ with Zariski-dense image, is given by 
\begin{equation}\label{dense_dim}
\dim \CalD = (2g-1) \dim G.
\end{equation}
Therefore it suffices to show that 
$$
(2g-1) \dim G - (2g-1) \dim K - 2g - \rank G   \geq 0
$$
This follows from Lemma \ref{codim_proper} provided 

\begin{align*}
(2g-1)(\dim G- \dim H) - 2g &\geq (2g-1) \rank G - 2g - \rank G\\
&= 2((g-1) \rank G - g )\\
&\geq 0,
\end{align*}
equivalently that $\rank G \geq g/(g-1)$. The lemma follows since the maximal value of $g/g-1$ is $2$ (recall $g \geq 2$.)

\end{proof}

Now we prove the main result of the paper.

\begin{thm}\label{main_result}
For every surface group $S_g$, where $g \geq 2$, and for any almost-simple real algebraic group $G$, 
\begin{equation*}
\dim   X_{S_g, G} =  (2g-1) \dim G .
\end{equation*}
\end{thm}

\begin{proof}
From Proposition \ref{proper_remove} it follows that we may restrict our attention to those irreducible components containing at least one representation $S_g \to G$ with Zariski-dense image. For such components the dimension is given by the lower bound estimate in the proof of Theorem \ref{main_result_section_3}.
\end{proof}

\section{Final Remarks}\label{finalremarks}

It remains to be seen whether the approach taken in this paper can be generalized to Fuchsian groups of genus $g \geq 0$. 
The approach taken in this paper essentially rests on finding a Zariski-dense homomorphism from a surface group to a semisimple real algebraic group. While the approach in this paper offers a \textit{precise formula} for the dimension of the representation variety $X_{\Gamma,G}$  but only for surface groups of genus $g \geq 2$, the results in \cite{LL} and \cite{Ki} offers an \textit{estimate} for all Fuchsian groups $g \geq 0$ but only for the representation variety $X_{\Gamma,G}^{\textrm{epi}}$ (\S \ref{intro}).

\end{document}